\documentclass[a4paper,12pt,twoside]{amsart}

 \usepackage[utf8]{inputenc}
 \usepackage[T1]{fontenc}
 \usepackage[normalem]{ulem}

\usepackage{amsfonts}
\usepackage{amsmath}
\usepackage{amssymb}

\newcommand{\C}{\mathbb{C}}

\newcommand{\g}{\mathfrak{g}}
\newcommand{\G}{\mathrm{G}}

\newcommand{\bdm}{\begin{displaymath}}
\newcommand{\edm}{\end{displaymath}}

\newcommand{\SU}{\mathrm{SU}}

\newcommand{\Sp}{\mathrm{Sp}}

\theoremstyle{definition}

\newtheorem{thm}{Theorem}[section]
\newtheorem{lem}[thm]{Lemma}
\newtheorem{defn}[thm]{Definition}
\newtheorem{prop}[thm]{Proposition}
\newtheorem{rem}[thm]{Remark}

\title[Extended solutions  in the special unitary group]{Extended solutions of the harmonic map equation in the special unitary group}

\author{Nuno Correia}
\address{Universidade da Beira Interior\\
Rua Marques d'Avila e Bolama, 6200-001 Covilha, Portugal}
\email{ncorreia@ubi.pt}

\author{Rui Pacheco}
\address{Universidade da Beira Interior\\
Rua Marques d'Avila e Bolama, 6200-001 Covilha, Portugal}
\email{rpacheco@ubi.pt}

\date{\today}

\thanks{The authors were partially supported by the Portuguese Government through FCT, under the project PEst-OE/MAT/UI0212/2011 (CMUBI)}


\begin{document}

\begin{abstract}

We classify all harmonic maps  of finite uniton number  from a Riemann surface into $\mathrm{SU}(n)$ in terms of certain pieces of the Bruhat decomposition of $\Omega_{\mathrm{alg}} \SU(n)$, the subgroup of algebraic loops in $\SU(n)$.  We give a description of the ``Frenet frame data" for such harmonic maps in a given class.
\end{abstract}


\maketitle

\section{Introduction}
Harmonic maps from a Riemann surface into a Lie group $\mathrm{G}$, with Lie algebra  $\mathfrak{g}$, correspond to certain holomorphic maps,
the \emph{extended solutions}, into the group $\Omega \mathrm{G}$
of based smooth loops in $\mathrm{G}$ \cite{uhlenbeck_1989}.
If the Fourier series associated to an extended solution $\Phi$ has finitely many terms, we say that $\Phi$ and the corresponding harmonic map  have \emph{finite uniton number}. It is well known that all harmonic maps from the two-sphere have finite uniton number \cite{uhlenbeck_1989}.

When $\mathrm{G}$ has trivial center, Burstall and Guest \cite{burstall_guest_1997} have classified harmonic maps with finite uniton number from a Riemann surface $M$ into $\mathrm{G}$ in terms of the pieces of the Bruhat decomposition of
$$\Omega_\mathrm{alg}\mathrm{G}=\{\gamma\in\Omega \mathrm{G} \,|\,\mbox{$\gamma$ and $\gamma^{-1}$ have finite Fourier series}\}.$$
 More precisely, each piece of the Bruhat decomposition coincides with the unstable manifold associated to the flow of the gradient vector field of a certain Morse-Bott function defined on the K\"{a}hler manifold $\Omega_\mathrm{alg}\mathrm{G}$; these unstable manifolds are parameterized by the elements of a certain integer lattice $\mathfrak{I}(\mathrm{G})$ in $\g$; any extended solution with finite uniton number takes values, off a discrete subset of $M$, in one of these unstable manifolds, and so corresponds to some element $\xi\in \mathfrak{I}(\mathrm{G})$; when $G$ has trivial center and  maximal torus with dimension $n$, there is a finite subset $\Xi_\mathrm{can}(G)$ of the integer lattice $\mathfrak{I}(\mathrm{G})$ with $2^n$ elements so that any harmonic map from $M$ to $G$ corresponds to an extended solution  with values, off a discrete subset, on the unstable manifold associated to some \emph{canonical element} $\xi\in \Xi_\mathrm{can}(G)$. Among such extended solutions, a distinguished type is that of $\mathrm{S}^1$-\emph{invariant} extended solutions, which correspond to harmonic maps admitting \emph{super-horizontal} holomorphic lifts into a certain \emph{twistor space}. For example, all harmonic spheres in $\mathrm{S}^n$ and $\C \mathrm{P}^n$ arise in this way (see \cite{eells_lemaire} and references therein).

In the present paper, we classify all  harmonic maps with finite uniton number from $M$ into the special unitary group $\mathrm{SU}(n)$ and corresponding inner symmetric spaces, the Grassmannians $\mathrm{Gr}(k,n)$ of $k$-dimensional subspaces of $\C^n$,  in terms of certain pieces of the Bruhat decomposition of $\Omega_{\mathrm{alg}} \SU(n)$.
For that we  use the results of \cite{correia_pacheco_3} in order to generalize the notion of canonical element of   $\mathfrak{I}(\mathrm{G})$ to the case where $G$ has not necessarily trivial center (recall that the center of $\SU(n)$ is isomorphic to the cyclic group $\mathbb{Z}_n$). Moreover, in the setting of the Grassmannian model for loops groups \cite{pressley_segal} we give a description of the ``Frenet frame data" for such harmonic maps in a given class. The Grassmannian model for loop groups  was exploited for the first time in the study of harmonic maps into the unitary group $\mathrm{U}(n)$ by Segal \cite{segal_1989}.  More recently, the Grassmannian model has been used in the study of harmonic maps into other  Lie groups and their inner symmetric spaces \cite{correia_pacheco_3,pacheco_2006,svensson_wood_2010}.
We remark that Ferreira, Sim\~{o}es and Wood \cite{ferreira_simoes_wood_2010} established an algebraic formula for all harmonic
maps with finite uniton number from a Riemann surface $M$ into the unitary
group $\mathrm{U}(n)$ in terms of freely chosen  meromorphic functions on $M$ and their derivatives. Since any such harmonic map has constant determinant, this formula can be easily applied in order to obtain all
harmonic maps with finite uniton number from $M$ into  $\mathrm{SU}(n)$. However, it does not clarifies how to choose the meromorphic functions in order to produce  harmonic maps associated to extended solutions in the class of a given  element $\xi\in  \Xi_\mathrm{can}(\mathrm{SU}(n))$.
In this paper we shall  see how to do that in the case of harmonic maps associated to $\mathrm{S}^1$-invariant extended solutions.

\section{Grassmannian model for loop groups}
Let us start by recalling from Pressley and Segal \cite{pressley_segal} some standard definitions and facts concerning the Grassmannian model for loop groups.

Fix on $\mathbb{C}^{n}$ the standard complex inner product $\langle
\cdot,\cdot\rangle$ and let $e_{1},\ldots,e_{n}$ be the standard basis vectors for
 $\mathbb{C}^{n}$. Given a complex subspace $V\subset \C^n$, we denote by $\pi_V$ the orthogonal projection onto $V$.
Let $H^n$ be the Hilbert space
 of square-summable $\C^{n}$-valued
functions on the
 circle and $ \langle\cdot,\cdot\rangle_H$ the induced complex inner product. This is the closed space generated by
the functions
$\lambda\mapsto\lambda^{i}e_{j}$, with $i\in\mathbb{Z}$ and $j=1,\ldots,n$.
 Consider the closed subspace $H_+^n$  of $H^n$ defined by
 $H_+^n=\mbox{Span}\{ \lambda^{i}e_{j}\,|\,i\geq 0,\,
 j=1,\ldots,n\}.$
 Let $\mbox{\emph{Grass}}(H^n)$ denote the set of all closed vector
 subspaces $W\subset H^n$ such that:
 the projection map $W\rightarrow H_+^n$ is
 Fredholm, and the projection map $W\rightarrow
 {H_+^n}^\perp$ is Hilbert-Schmidt;
the images of the projection maps $W^{\perp}\rightarrow H_+^n$, $W\rightarrow
 {H_+^n}^\perp$ are contained in $C^{\infty}(\mathrm{S}^{1};\mathbb{C}^{n})$.  Define
$$\mathrm{Gr}^n=\{W\in \mbox{\emph{Grass}}(H^n)\,|\, \lambda W\subseteq
W \}.$$
The action of the infinite-dimensional Lie group
 $\Lambda \mathrm{U}(n)=\big\{\gamma:S^1\to \mathrm{U}(n)\,|\, \mbox{$\gamma$ is smooth}\big\}$
 on $\mathrm{Gr}^n$ defined by $\gamma W=\{\gamma f\,|\,f\in W\}$ is transitive. By considering Fourier series, it is easy to see that the
 isotropy subgroup at $H_+^n$ is precisely
 ${\mathrm{U}}(n)$. Hence
 $\mathrm{Gr}^n \cong \Lambda \mathrm{U}(n)/\mathrm{U}(n)\cong \Omega \mathrm{U}(n).$  This homogeneous space carries a natural invariant structure of K\"{a}hler manifold.

\begin{rem}\label{rem} Given $W\in \mathrm{Gr}^n$, then
$\dim W \ominus \lambda W=n$, where $W \ominus
\lambda W$ denotes the orthogonal complement of $\lambda W$ in
$W$, and the evaluation map $\rm{ev}_\lambda:W\ominus \lambda W:\to \C^n$ at $\lambda\in\mathrm{S}^1$ is a unitary isomorphism. If we choose an orthonormal basis for $W\ominus\lambda W$,
$\{w_1,\ldots,w_n\}$, we can put the vector-valued functions
$w_i$ side by side to form an $(n\times n)$-matrix valued
function $\gamma$  on $\mathrm{S}^1$, that is, a loop $\gamma\in\Lambda
{\mathrm{U}}(n)$. It can be shown  that $W=\gamma  H^n_+$.
\end{rem}

A loop
 $\gamma \in\Omega \mathrm{U}(n)$ is said to be \emph{algebraic} if both $\gamma$ and $\gamma^{-1}$ have finite
Fourier series. Denote by
$\Omega_{\rm{alg}}\mathrm{U}(n)$ the subgroup of algebraic loops. This
subgroup acts on
$$\mathrm{Gr}^n_{\mathrm{alg}}=\{W \in \mathrm{Gr}^n\,|\,
\lambda^k H_+^n \subseteq W\subseteq \lambda^{-k}
H_+^n\,\,\textrm{for some } k\in\mathbb{N} \},
$$
and we have  $\mathrm{Gr}^n_{\mathrm{alg}}\cong
 \Omega_{\mathrm{alg}}{\mathrm{U}}(n)$. Given $r\leq k$, we set $$\Omega_r^k\mathrm{U}(n)=\big\{\gamma\in \Omega_{\mathrm{alg}} \mathrm{U}(n)\,|\,\gamma(\lambda)=\sum_{i=r}^{k}\gamma_i\lambda^i\big\},$$  where the coefficients $\gamma_i$ are constant $(n\times n)$-complex matrices.
 If $\mathrm{G}$ is a subgroup of $\mathrm{U}(n)$, we shall denote by $\mathrm{Gr}(\mathrm{G})$ the subspace of $\mathrm{Gr}^n$ that corresponds to $\Omega \mathrm{G}$ and by ${\mathrm{Gr}}_{\mathrm{alg}}(\mathrm{G})$ the subspace of $\mathrm{Gr}(\mathrm{G})$ that corresponds to $\Omega_{\mathrm{alg}}\mathrm{G}$.

 \section{The Bruhat Decomposition of $\mathrm{Gr}_{\mathrm{alg}}(\mathrm{G})$}
Next we describe the Bruhat decomposition for algebraic loop groups. For more details we refer the reader to \cite{burstall_guest_1997} and \cite{pressley_segal}.

Consider  a compact matrix semi-simple Lie group $\mathrm{G}$. Fix a maximal torus $\mathrm{T}$ of $\mathrm{G}$ with Lie algebra $\mathfrak{t}\subset \mathfrak{g}$, let $\Delta\subset \sqrt{-1} \mathfrak{t}^*$ be the corresponding
set of roots and  denote by $\g_\alpha$ the root space of  $\alpha\in \Delta$. The integer lattice $\mathfrak{I}(\mathrm{G}) =
(2\pi)^{-1} \exp^{-1}(e)\cap \mathfrak{t}$ may be identified with the group of homomorphisms $\mathrm{S}^1\to \mathrm{T}$, by
associating to $\xi\in \mathfrak{I}(\mathrm{G})$ the homomorphism
$\gamma_\xi$ defined by $\gamma_\xi(\lambda)=\exp{(-\sqrt{-1}\ln(\lambda)\xi)}$.
Let  $H_1,\ldots,H_k\in \mathfrak{t}$ be dual to the positive simple roots $\alpha_1,\ldots,\alpha_k\in \Delta^+$ of $\g^\C$:
$\alpha_i(H_j)=\sqrt{-1}\,\delta_{ij}$. By applying the well-known formula  $\mathrm{Ad}(\exp(\eta))=\exp(\mathrm{ad}(\eta))$, for all $\eta\in\g^\C$, we can easily check that
the integer lattice  $\mathfrak{I}(G)$ is contained in $\mathbb{Z}H_1\oplus\ldots\oplus \mathbb{Z}H_k$.
 Denote by $\g^\xi_i$ the $\sqrt{-1}\,i$-eigenspace of $\mathrm{ad}{\xi}$, with $i\in\mathbb{Z}$.
We have on $\mathfrak{g}^\C$ the structure of graded Lie algebra:
\begin{equation*}
\g^\C=\!\!\!\bigoplus_{i\in\{-r(\xi)\ldots,r(\xi)\}}\!\!\!\mathfrak{g}^\xi_i,\quad [\mathfrak{g}^\xi_i, \mathfrak{g}^\xi_j]\subset \mathfrak{g}^\xi_{i+j},
\end{equation*}
 where $r(\xi)=\mathrm{max}\{i\,|\,\,\g_i^\xi\neq 0\}$, and
 \begin{equation}\label{gis}
\g_i^\xi=\!\!\bigoplus_{\alpha(\xi)=\sqrt{-1}\,i}\!\!\g_\alpha.
\end{equation}

Set $\Lambda^+\mathrm{G}^\C=\{\gamma:S^1\to \mathrm{G}^\C\,|\,\,\mbox{$\gamma$ extends holomorphically for $|\lambda|\leq 1$}\}$.
For each
$\xi\in \mathfrak{I}(\mathrm{G})$, we write
$\Omega_\xi=\{g\gamma_\xi g^{-1} \,|\,\, g \in \mathrm{G}\},$ the conjugacy class of homomorphisms $\mathrm{S}^1\to \mathrm{G}$ which contains $\gamma_\xi$.
This is a complex homogeneous space:
$$\Omega_\xi\cong \mathrm{G}^\C\big/\mathrm{P}_\xi,\,\mbox{with}\,\mathrm{P}_\xi=\mathrm{G}^\C\cap \gamma_\xi\Lambda^+\mathrm{G}^\C\gamma_\xi^{-1}.$$
Taking account that
  $\gamma_\xi X_j\gamma_\xi^{-1}=\lambda^jX_j$ for each $X_j\in \g^\xi_j$ (this is a direct consequence of the formula $\mathrm{Ad}(\exp(\eta))=\exp(\mathrm{ad}(\eta))$, for all $\eta\in\g^\C$), one can easily check  that the Lie algebra of the isotropy subgroup $\mathrm{P}_\xi$ is precisely the parabolic subalgebra $\mathfrak{p}_\xi=\bigoplus_{i\leq 0}\g^\xi_i$ induced by $\xi$.

Now, fix a positive set of roots $\Delta^+\subset \Delta$ and set $\mathfrak{I}'(\mathrm{G})=\{\xi\in\mathfrak{I}(\mathrm{G})|\, \mbox{$\alpha(\xi)\geq 0$ for all $\alpha\in \Delta^+$}\}.$ We have:

\begin{thm}\cite{pressley_segal}
 \emph{Bruhat decomposition: }$\mathrm{Gr}_\mathrm{alg}(\mathrm{G})=\bigcup_{\xi\in \mathfrak{I}'(\mathrm{G})}\Lambda^+_\mathrm{alg}\mathrm{G}^\C\gamma_\xi  H_+^n$.
\end{thm}

Define $U_\xi(\mathrm{G})\subset \Omega_{\mathrm{alg}}\mathrm{G}$   by  $U_\xi(\mathrm{G})  H_+^n=\Lambda^+_{\mathrm{alg}}\mathrm{G}^\C\gamma_\xi H^n_+.$ This is also a  complex homogeneous space
of $\Lambda^+_{\mathrm{alg}}\mathrm{G}^\C$
with isotropy subgroup at $\gamma_\xi$ given by $\Lambda^+_{\mathrm{alg}}\mathrm{G}^\C\cap \gamma_\xi \Lambda^+\mathrm{G}^\C \gamma_\xi^{-1}.$ Moreover,  $U_\xi(\mathrm{G})$ carries the structure of holomorphic vector bundle over $\Omega_\xi$ whose bundle map
$u_\xi:U_\xi(\mathrm{G})\to \Omega_\xi$ is precisely the natural map
  $[\gamma]\mapsto [\gamma(0)]$.  The holomorphic tangent bundle of $U_\xi(\mathrm{G})$ is given by
\begin{equation}\label{holm}
T^{1,0}U_\xi(\mathrm{G})\cong \Lambda^+_{\mathrm{alg}}\mathrm{G}^\C\times_{\Lambda^+_{\mathrm{alg}}\mathrm{G}^\C\cap \gamma_\xi\Lambda^+\mathrm{G}^\C\gamma_\xi^{-1}}\Lambda^+_{\mathrm{alg}}\mathfrak{g}^\C\big/\Lambda^+_{\mathrm{alg}}\mathfrak{g}^\C\cap \gamma_\xi\Lambda^+\mathfrak{g}^\C\gamma_\xi^{-1}
\end{equation}

In terms of the Grassmannian model, the  bundle map $u_\xi:U_\xi(\mathrm{G})\to\Omega_\xi$ can be described as follows.
 Take $\gamma\in U_\xi(\mathrm{G})$ and $W=\gamma H^n_+\in \mathrm{Gr}_{\mathrm{alg}}(\mathrm{G})$, with
  $\lambda^rH^n_+\subset W\subset\lambda^{-s}H^n_+$.
Fix $\Psi\in \Lambda^+_{\mathrm{alg}}\mathrm{G}^\C$ such that
$W=\Psi\gamma_\xi H^n_+$. Write
\begin{equation*}\label{flag}
\gamma_\xi H^n_+=\lambda^{-s}A^\xi_{-s}+\ldots+\lambda^{r-1}A^\xi_{r-1}+\lambda^rH^n_+,
\end{equation*} where the subspaces $A^\xi_i$ define a flag
\begin{equation}\label{flag1}
\{0\}=A^\xi_{-s-1}\subsetneq A^\xi_{-s}\subseteq A^\xi_{-s+1}\subseteq \ldots\subseteq A^\xi_{r-1}\subsetneq A^\xi_r=\C^n.
\end{equation}
Set $A_i=\Psi(0)A^\xi_i= p_i(W\cap\lambda^iH^n_+),$ where $p_i: H^n \to\C^n$ is defined by
\begin{equation}\label{pi}
p_i(\sum\lambda^ja_j)=a_i.
\end{equation} Then
\begin{equation}\label{popo1}
u_\xi(W)=\lambda^{-s}A_{-s}+\ldots+\lambda^{r-1}A_{r-1}+\lambda^rH^n_+.
\end{equation}

%

Following \cite{correia_pacheco_3}, consider the partial order $\preceq$ over $\mathfrak{I}'(\mathrm{G})$ defined by:
$\xi\preceq \xi'$ if $\mathfrak{p}^{\xi}_i\subseteq \mathfrak{p}^{\xi'}_i$
for all $i\geq 0$, where $\mathfrak{p}_i^\xi=\bigoplus_{j\leq i}\g_j^\xi$.
Given two elements $\xi,\xi'\in \mathfrak{I}'(\mathrm{G})$ such that $\xi\preceq \xi'$, it can be shown \cite{correia_pacheco_3} that
   $$\Lambda^+_{\mathrm{alg}}\mathrm{G}^\C\cap \gamma_\xi \Lambda^+\mathrm{G}^\C \gamma_\xi^{-1}\subseteq \Lambda^+_{\mathrm{alg}}\mathrm{G}^\C\cap \gamma_{\xi'} \Lambda^+\mathrm{G}^\C \gamma_{\xi'}^{-1}.$$
This allows us to define, for $\xi\preceq \xi'$, a $\Lambda^+_{\mathrm{alg}}\mathrm{G}^\C$-invariant fibre bundle morphism
$\mathcal{U}_{\xi,\xi'}:U_\xi\to U_{\xi'}$  by
\begin{equation*}\label{uxi}
\mathcal{U}_{\xi,\xi'}(\Psi\gamma_{\xi}H^n_+)=\Psi\gamma_{\xi'}H^n_+, \quad \Psi\in\Lambda^+_{\mathrm{alg}}\mathrm{G}^\C.\end{equation*}
Since the holomorphic structures on  $U_\xi(\mathrm{G})$ and $U_{\xi'}(\mathrm{G})$ are induced by the holomorphic structure on $\Lambda^+_{\mathrm{alg}}\mathrm{G}^\C$, the fibre-bundle morphism  $\mathcal{U}_{\xi,\xi'}$ is holomorphic.

\section{Harmonic maps into a Lie group}

\subsection{Extended solutions}

Let $M$ be a Riemann surface and  $\varphi:M\rightarrow \mathrm{G}\subseteq \mathrm{U}(n)$  a map into
a compact matrix Lie group. Equip $\mathrm{G}$ with a bi-invariant metric.
Define $\alpha=\varphi^{-1}{d}\varphi$ and let $\alpha=\alpha'+\alpha''$
be the type decomposition of $\alpha$ into $(1,0)$ and
$(0,1)$-forms. It is well known \cite{uhlenbeck_1989}
 that $\varphi:M\rightarrow \mathrm{G} $ is harmonic if and only if the loop of
$1$-forms given by
\begin{equation}
\label{flcon}
 \alpha_\lambda=\frac{1-\lambda^{-1}}{2}
\alpha'+\frac{1-\lambda}{2} \alpha''
 \end{equation}
 satisfies the Maurer-Cartan equation ${d}\alpha_\lambda + \frac{1}{2}[\alpha_\lambda\wedge \alpha_\lambda]=0$
 for each $\lambda\in \mathrm{S}^1$.
Then, if $M$ is simply connected and $\varphi$ is harmonic, we can
integrate to obtain a map $\Phi:M\rightarrow \Omega \mathrm{G}$, the \textit{extended solution} associated to $\varphi$, such that
$\alpha_\lambda=\Phi_\lambda^{-1}{d}\Phi_\lambda$ and $\Phi_{-1}=\varphi$. Conversely, if $\Phi:M\rightarrow \Omega \mathrm{G}$ is an extended solution, that is if it integrates a loop of $1$-form of the form \eqref{flcon}, then $\varphi=\Phi_{-1}:M\rightarrow \mathrm{G}$ is harmonic.


\begin{thm}\cite{burstall_guest_1997}\label{usd}
\emph{Let $\Phi:M\to \Omega_{\mathrm{alg}}\mathrm{G}$ be an extended solution. Then there exists some $\xi\in \mathfrak{I}'(\mathrm{G})$, and some discrete subset $D$ of $M$, such that $\Phi(M\setminus D)\subseteq U_\xi(\mathrm{G})$.}
  \end{thm}

Given a smooth map $\Phi:M\setminus D\to U_\xi(\mathrm{G})$, consider $\Psi:M\setminus D \to \Lambda_{\mathrm{alg}}^+\mathrm{G}^\C$ such that $\Phi H^n_+=\Psi\gamma_\xi H^n_+$. Clearly,
 $\Psi\gamma_\xi=\Phi b$ for some $b:M\setminus D\to \Lambda^+_{\mathrm{alg}}G^\C.$
Write
\begin{equation*}\label{not}
\Psi^{-1}\Psi_z=\sum_{i\geq 0} X'_i\lambda^i,\,\,\,\,\Psi^{-1}\Psi_{\bar{z}}=\sum_{i\geq 0} X''_i\lambda^i.\end{equation*}
Proposition 4.4 in  \cite{burstall_guest_1997} establishes that   $\Phi$ is an extended solution if, and only if,
\begin{equation}\label{im}
\mathrm{Im} X'_i\subset \,\mathfrak{p}^\xi_{i+1},\,\,\,\,\mathrm{Im} X''_i\subset \mathfrak{p}^\xi_{i},
\end{equation}
where $\mathfrak{p}_i^\xi=\bigoplus_{j\leq i}\g_j^\xi.$ The second condition says that $\Phi:M\setminus D\to U_\xi(\mathrm{G})$ is holomorphic.

The bundle morphism $\mathcal{U}_{\xi,\xi'}$ and the bundle map $u_\xi$ are well behaved with respect to extended solutions:

\begin{thm}\label{popo}\cite{correia_pacheco_3}
\emph{Given an extended solution $\Phi:M\setminus D\to U_\xi(\mathrm{G})$ and an element $\xi'\in \mathfrak{I}'(\mathrm{G})$ such that $\xi\preceq \xi'$, then
$\mathcal{U}_{\xi,\xi'}(\Phi)=\mathcal{U}_{\xi,\xi'}\circ \Phi:M\setminus D\to U_{\xi'}(\mathrm{G})$ is a new extended solution.}
 \end{thm}
\begin{thm}\cite{burstall_guest_1997}
\emph{If  $\Phi:M\setminus D\to U_\xi(\mathrm{G})$ is an extended solution, then  $u_\xi\circ\Phi:M\setminus D\to \Omega_\xi$ is an extended solution.}
\end{thm}

An \emph{$\mathrm{S}^1$-invariant extended solution} is an extended solution   which takes values in $\Omega_\xi$, for some $\xi\in \mathfrak{I}'(\mathrm{G})$.

\subsection{Harmonic maps into inner $\mathrm{G}$-symmetric spaces}

Given a compact  (connected) Lie group $\mathrm{G}$,
 each connected component of $\sqrt{e}=\{g\in
\mathrm{G}\,|\,\,g^2=e\}$ is a compact inner symmetric space \cite{burstall_guest_1997}. Conversely, any compact (connected) inner  $\mathrm{G}$-symmetric space may be immersed in $\mathrm{G}$ as a connected component of $\sqrt{e}$.
 Moreover,
the embedding of each component of $\sqrt{e}$ in $\mathrm{G}$ is totally
geodesic. Hence harmonic maps into $\mathrm{G}$-inner symmetric spaces can be viewed as special harmonic maps into $\mathrm{G}$.

As in \cite{burstall_guest_1997},  define the involution $\mathcal{I}:\Omega \mathrm{G}  \rightarrow \Omega \mathrm{G}$ by $\mathcal{I}(\gamma)(\lambda)
=\gamma(-\lambda)\gamma(-1)^{-1}.$ Write $$\Omega^\mathcal{I}
\mathrm{G}=\{\gamma\in \Omega \mathrm{G}\,|\,\,\mathcal{I}(\gamma)=\gamma\}$$
 for the fixed set of $\mathcal{I}$. Let $M$ be a Riemann surface and
 $\Phi:M\rightarrow \Omega^\mathcal{I} \mathrm{G}$ an extended solution. Then
 $\varphi=\Phi_{-1}$ defines a harmonic map from $M$ into a
 connected component of $\sqrt{e}$.
Conversely, if $\varphi:M\rightarrow\sqrt{e}$ is a harmonic map,  there exists an extended
 solution $\Phi:M\rightarrow \Omega^\mathcal{I} \mathrm{G}$ such that
 $\varphi=\Phi_{-1}$. Under the identification $\Omega \mathrm{G}\cong  \mathrm{Gr}(\mathrm{G})$,
$\mathcal{I}$ induces
 an involution on $\mathrm{Gr}(\mathrm{G})$, that we shall also denote by $\mathcal{I}$,
 and $\Omega^\mathcal{I} \mathrm{G}$ can be identified with
 \bdm
\mathrm{Gr}^\mathcal{I}(\mathrm{G})=\{W\in \mathrm{Gr}(\mathrm{G})\,|\,\,
 \,\mbox{if $s(\lambda)\in W$ then $s(-\lambda)\in W$}\}.
 \edm
Corresponding to the extended solution $\Phi:M\rightarrow \Omega^\mathcal{I}
\mathrm{G}$, consider $W=\Phi H_+:M \rightarrow
\mathrm{Gr}^\mathcal{I}(\mathrm{G})$.

For each $\xi \in \mathfrak{I}'(\mathrm{G})$ we can associate the symmetric space $N_\xi=\{g\gamma_\xi(-1)g^{-1}\,|\,\,g\in G$\}.
If an extended solution takes values in $U_\xi^{\mathcal{I}}(\mathrm{G})=U_\xi(\mathrm{G})\cap \Omega^{\mathcal{I}}\mathrm{G}$, then the corresponding harmonic map takes values in $N_\xi$.
Observe that, for $\xi$ and $\xi'$ in  $\mathfrak{I}'(\mathrm{G})$, if $\xi-\xi'\in\mathfrak{I}^{2}(\mathrm{G}):=\pi^{-1}\exp^{-1}(e)\cap \mathfrak{t}$, then $N_\xi=N_{\xi'}$.
Moreover, as shown in \cite{correia_pacheco_3},
 if $\xi\preceq \xi'$, then $\mathcal{U}_{\xi,\xi'}(U_\xi^{\mathcal{I}}(\mathrm{G}))\subset U_{\xi'}^{\mathcal{I}}(\mathrm{G})$. To sum up, if we define a new partial order $\preceq_\mathcal{I}$ in $\mathfrak{I}'(\mathrm{G})$ by
\begin{equation*}
 \xi\preceq_\mathcal{I}\xi'\,\,\, \mbox{ if}\,\,\, \xi\preceq \xi' \,\,\, \mbox{and}\,\,\, \xi-\xi'\in\mathfrak{I}^{2}(\mathrm{G}),
\end{equation*}
 the following holds:
\begin{prop}\label{proposition}
 \emph{ If $\xi\preceq_\mathcal{I} \xi'$, then $\mathcal{U}_{\xi,\xi'}(U_\xi^{\mathcal{I}}(\mathrm{G}))\subset U_{\xi'}^{\mathcal{I}}(\mathrm{G})$ and $N_\xi=N_{\xi'}$.}
\end{prop}

\subsection{Extended solutions from the Grassmannian point of view}

 Let $W:{M} \rightarrow \mathrm{Gr}(\mathrm{G})$ correspond to a smooth map $\Phi:M\to \Omega \mathrm{G}$ under the
identification $\Omega \mathrm{G} \cong \mathrm{Gr}(\mathrm{G})$, that is $W=\Phi
H^n_+$.
Segal \cite{segal_1989} has observed that $\Phi$ is an extended solution if, and only if, $W$ is a
solution of equations:
\begin{align}
W_z  &\subseteq  {\lambda}^{-1}W \label{phh},\\  W_{\bar{z}}
 & \subseteq W.\label{hh}
\end{align}
Condition \eqref{phh} means that, in any local complex coordinate  $z$, $\frac{\partial s}{\partial z}(z)$
is contained in the subspace $\lambda^{-1}W(z)$ of $H^n$, for every
(smooth) map  $s : M\rightarrow H^n$ such that $s(z)\in
W(z)$. Inspired by \cite{burstall_rawnsley_1990} (Section F of Chapter 8), we call \eqref{phh} the \emph{pseudo-horizontality} condition. Condition \eqref{hh} is interpreted in a similar way and states that $W$ is a holomorphic vector subbundle of $M\times H^n$.

\begin{rem}
  Consider some discrete set $D\subset M$, an element $\xi\in \mathfrak{I}'(\mathrm{G})$ and an extended solution  $\Phi:M\setminus D\to U_\xi(\mathrm{G})$. As explained in Remark 2.5 of \cite{correia_pacheco_2}, the bundle
  $W=\Phi H^n_+$ can be extended holomorphically to $M$, and, consequently,
$\Phi$ defines a global extended solution from  $M$ to $\Omega_{\mathrm{alg}}\mathrm{G}$.
\end{rem}

If $\Phi:M\setminus D\to U_\xi(\mathrm{G})$ is an extended solution and $W=\Phi H^n_+$, then $u_\xi(W)=u_\xi\circ\Phi H_+^n$ is given pointwise by
(\ref{popo1}) and we get holomorphic subbundles $A_i$ of the trivial bundle $\underline{\C}^n=M\times \C^n$ such that
\begin{equation}\label{super}
0\subsetneq A_{-s} \subseteq \ldots \subseteq A_{r-1} \subsetneq A_r=\underline{\C}^n.
\end{equation}
 The pseudo-horizontally condition  implies that ${A_i}_z\subseteq A_{i+1}$, that is, following again the terminology of
\cite{burstall_rawnsley_1990}, the flag of holomorphic vector bundles \eqref{super} is \emph{super-horizontal}.

\subsection{Normalization of extended solutions}

The following theorem, which is a  generalization of Theorem 4.5 in \cite{burstall_guest_1997}, is fundamental to the classification of extended solutions.

\begin{thm}\label{nor}\cite{correia_pacheco_3}
\emph{ Let $\Phi:M\setminus D\to U_\xi(\mathrm{G})$ be an extended solution. Take $\xi'\in \mathfrak{I}'(\mathrm{G})$ such that $\xi\preceq {\xi'}$ and $\g_0^\xi=\g_0^{\xi'}$. Then there exists some constant loop $\gamma\in \Omega_{\mathrm{alg}}\mathrm{G}$ such that $\gamma\Phi:M\setminus D\to U_{\xi'}(\mathrm{G})$.}
\end{thm}

A similar statement holds for extended solutions associated to harmonic maps into symmetric spaces:

\begin{thm}\label{nor2} \emph{ Let $\Phi:M\setminus D\to U^{\mathcal{I}}_\xi(\mathrm{G})$ be an extended solution. Take $\xi'\in \mathfrak{I}'(\mathrm{G})$ such that $\xi\preceq_\mathcal{I} {\xi'}$. Then there exists some constant loop $\gamma\in \Omega^{\mathcal{I}}_{\mathrm{alg}}\mathrm{G}$ such that
$\gamma\Phi:M\setminus D\to U^{\mathcal{I}}_{\xi'}(\mathrm{G})$.}
\end{thm}
\begin{proof}
We can write $\Phi H_+^n=\Psi \gamma_\xi H_+^n$, where $\Psi:M\setminus D\to  \Lambda_{\mathrm{alg}}^+G^\C$ contains only even powers of $\lambda$, and consequently   $\Psi^{-1}\Psi_z=\sum_{i\geq 0}X'_i\lambda^i$ contains only even powers of $\lambda$. The extended solution equation \eqref{im} gives $\mathrm{Im}\,X'_{2j} \subset \mathfrak{p}^{\xi}_{2j+1}$ for all $j\geq 0$.  Set $\hat{\xi}=\xi-\xi'\in \mathfrak{I}^2(\mathrm{G})$. Clearly $\xi\preceq \hat{\xi}$, hence $\mathfrak{p}^\xi_{2j+1}\subseteq  \mathfrak{p}^{\hat{\xi}}_{2j+1}$ for all $j\geq 0$. On the other hand, since $\alpha(\hat{\xi})=2\sqrt{-1}\mathbb{Z}$ for any positive root $\alpha$, we have  $\mathfrak{g}^{\hat{\xi}}_{2j+1}=0$  and, consequently,  $\mathfrak{p}^{\hat{\xi}}_{2j+1}=\mathfrak{p}^{\hat{\xi}}_{2j}$. Hence,
$$\mathrm{Im}\,\Psi^{-1}\Psi_z \subseteq \bigoplus_{j\geq 0} \mathfrak{p}^\xi_{2j+1}\lambda^{2j}\subseteq \bigoplus_{j\geq 0} \mathfrak{p}^{\hat{\xi}}_{2j}\lambda^{2j}\subseteq \Lambda^{+}_{\mathrm{alg}}\mathfrak{g}^\C\cap \gamma_{\hat{\xi}}\Lambda^+\mathfrak{g}^\C\gamma_{\hat{\xi}}^{-1}.$$ Taking account \eqref{holm}, we conclude that $\mathcal{U}_{\xi,\hat{\xi}}(\Phi)$ is anti-holomorphic.
On the other hand, since any extended solution is holomorphic and $\Phi$ is an extended solution,  Theorem \ref{popo} asserts that
 $\mathcal{U}_{\xi,\hat{\xi}}(\Phi)$ is also holomorphic. Being both holomorphic and anti-holomorphic, it must be equal to a constant loop $\gamma^{-1}$. By   Proposition \ref{proposition} we have $\gamma^{-1} \in \Omega^{\mathcal{I}}_{\mathrm{alg}} \mathrm{G}$. Write $\Psi \gamma_{\hat{\xi}}=\gamma^{-1}b$, for some map $b:M\to \Lambda^+\mathrm{G}$.
Then
$$\Phi H^n_+=\Psi \gamma_\xi H^n_+= \gamma^{-1}b  \gamma_{\hat{\xi}}^{-1}\gamma_\xi H^n_+= \gamma^{-1}b  \gamma_{\xi'} H^n_+,$$
which implies that $\gamma \Phi$ takes values in $U^{\mathcal{I}}_{\xi'}(\mathrm{G})$.
\end{proof}

Given $\xi=\sum n_iH_i$ and $\xi'=\sum n'_iH_i$ in $\mathfrak{I}'(\mathrm{G})$,  we have $n_i,n'_i\geq 0$ and observe that  $\xi\preceq\xi'$ if and only if $n'_i\leq n_i$ for all $i$.
For each $I\subseteq \{1,\ldots,k\}$, define the cone
$$\mathfrak{C}_{I}=\Big\{\sum_{i=1}^k n_i H_i|\, n_i\geq 0, \,\mbox{$n_j=0$ iff $j\notin I$}\Big\}.$$

\begin{defn}
   \emph{Let $\xi\in\mathfrak{I}'(\mathrm{G})\cap \mathfrak{C}_{I}$. We say that $\xi$ is a \emph{$I$-canonical element} of $\mathfrak g$ with respect to  $\Delta^+$ if it is a maximal element of $(\mathfrak{I}'(\mathrm{G})\cap \mathfrak{C}_{I},\preceq)$, that is,  if $\xi\preceq \xi'$ and $\xi'\in \mathfrak{I}'(\mathrm{G})\cap \mathfrak{C}_{I}$ then $\xi=\xi'$. Similarly, we say that  $\xi$ is a \emph{symmetric canonical element} of $\mathfrak g$ with respect to  $\Delta^+$ if it is a maximal element of $(\mathfrak{I}'(\mathrm{G}),\preceq_{\mathcal {I}})$ }
\end{defn}

 When $G$ has trivial center, which is the case considered in \cite{burstall_guest_1997}, the duals $H_1,\ldots,H_k$ belong to the integer lattice. Then, for each $I$ there exists a unique $I$-canonical element, which is given by $\xi_I=\sum_{i\in I}H_i$.

\begin{thm}
\emph{  Let  $\Phi:M \to \Omega_{\rm{alg}}\mathrm{G}$ be an extended solution. There exist a constant loop $\gamma\in \Omega_{\rm{alg}}\mathrm{G}$, a subset $I\subseteq \{1,\ldots,k\}$, a $I$-canonical element $\xi'$ and a discrete subset $D\subset M$, such that
$\gamma\Phi(M\setminus D)\subseteq U_{\xi'}(\mathrm{G})$.}
\end{thm}
\begin{proof}
Take $D\subset M$ and $\xi\in \mathfrak{I}'(\mathrm{G})$ in the conditions of Theorem \ref{usd}. Write $\xi=\sum_{i=1}^kn_iH_i$, with $n_i\geq 0$, and set $I=\{i| n_i>0\}$. By Zorn's lemma,  there certainly exists a $I$-canonical element $\xi'$ such that $\xi\preceq \xi'$. On the other hand, from \eqref{gis} we see that $\g_0^\xi=\g_0^{\xi'}$. Hence the result follows from Theorem \ref{nor}. \end{proof}

\begin{thm}
\emph{  Let  $\Phi:M \to \Omega^{\mathcal{I}}_{\rm{alg}}\mathrm{G}$ be an extended solution with values in $U_{\xi}^{\mathcal{I}}(\mathrm{G})$, for some $\xi\in\mathfrak{I}'(\mathrm{G})$, off a discrete set $D$. There exist a constant loop $\gamma\in \Omega^{\mathcal{I}}_{\rm{alg}}\mathrm{G}$ and a  symmetric canonical element $\xi'$  such that
$\gamma\Phi(M\setminus D)\subseteq U^{\mathcal{I}}_{\xi'}(\mathrm{G})$ and $N_\xi=N_{\xi'}$.}
\end{thm}
\begin{proof}
 By Zorn's lemma,  there certainly exists a  symmetric canonical element $\xi'$ such that $\xi\preceq_\mathcal{I} \xi'$. The result follows from Proposition \ref{proposition} and Theorem \ref{nor2}. \end{proof}

\subsection{Frenet frame data for extended solutions into $\Omega_{\mathrm{alg}}\mathrm{U}(n)$}\label{construction}

Given a finite collection $\{s_j\}$ of meromorphic sections of the trivial bundle $\underline{\C}^n=M\times \C^n$, we obtain an holomorphic bundle away from a discrete subset of $M$, and we can fill in holes to extend it  to subbundle $E$ over $M$ of $\underline{\C}^n$. In this case, we denote $E=\mathrm{Span}\{s_j\}$. Reciprocally, any holomorphic subbundle $E$ of $\underline{\C}^n$ has a global meromorphic frame $\{s_1,\ldots, s_k\}$, with $k=\mathrm{rank}\, E$, as explained in  \cite{svensson_wood_2010}. For $i>0$, the \emph{$(i)$-th osculating bundle} of $E$ is the subbundle $E^{(i)}$ of $\underline{\C}^n$ spanned by the local holomorphic sections of $E$ and their derivatives up to $i$. We also define the \emph{$(-i)$-th osculating bundle} of $E$ as the subbundle $E^{(-i)}$ of $\underline{\C}^n$ spanned by the local holomorphic sections  of $E$ whose derivatives up to $i$ are also local sections of $E$.
 Let  $g_E=\mathrm{rank}\,E^{(1)}-\mathrm{rank}\,E$  and $r_E$ be the remainder of the positive integer division of $\mathrm{rank}\,E$ by $g_E$: $\mathrm{rank}\,E=q_Eg_E+r_E$.

As Guest \cite{guest_2002} has observed, any smooth map $W:M\to \mathrm{Gr}^n$  corresponding to an extended solution  $\Phi:M\to \Omega_{r}^k\mathrm{U}(n)$ is \emph{generated} by a certain holomorphic subbundle $X$, a \emph{Frenet frame} of $\Phi$,
of the trivial bundle $ M\times \lambda^{r}H_+\big/\lambda^kH_+$ by setting
\begin{equation}\label{frenetframe}
W=X+\lambda X^{(1)}+\ldots+\lambda^{k-r-1}X^{(k-r-1)}+\lambda^{k}H_+.
\end{equation}
Hence any extended solution $\Phi:M\to \Omega_{\mathrm{alg}}\mathrm{U}(n)$ can be obtained by applying a finite number of algebraic operations on sets of meromorphic functions on $M$, since $X$ can be chosen arbitrarily. In \cite{ferreira_simoes_wood_2010,svensson_wood_2010} the authors established explicit algebraic formulae relating  Frenet frames $X$ with different classes of uniton factorizations of harmonic maps. Next we will  give a description of the Frenet frames associated to  extended solutions with values in a fixed piece $U_\xi(\mathrm{U}(n))$ of the Bruhat decomposition of $\Omega_{\mathrm{alg}}\mathrm{U}(n)$ and we  establish a pure algebraic method to obtain all $\mathrm{S}^1$-invariant extended solutions with values in a fixed $\Omega_\xi$.

Choose a local complex coordinate $z$ and a local section $s$ of $E$. Differentiating $\pi^\perp_E(s)=0$, where  $\pi^\perp_E$ is the orthogonal projection onto $E^\perp$, we get $\pi^\perp_E(s_z)=-(\pi_E^\perp)_z(s)$. Hence the association $s \mapsto \pi^\perp_E(s_z)$ defines a local vector bundle morphism $\mathcal{A}_E:E\to E^\perp$, which, following \cite{burstall_wood_1986}, we call the \emph{$\partial'$-second fundamental form}  of $E$ in $\C^n$, whose kernel and image do not depend on the choice of the local coordinate $z$. It follows from the linearity of the $\partial'$-second fundamental form that:

 \begin{lem}\label{E}
   \emph{Let $E$ be a holomorphic vector subbundle of $\underline{\C}^n$.
   \begin{enumerate}\item[a)] For all $i\geq 1$, $E^{(-i)}=\ker \mathcal{A}_{E^{(-i+1)}} $ is locally spanned by those sections $s$ of $E$ solving the following system of algebraic linear equations: $(\pi_{E^{(-j)}}^\perp)_z(s)=0$ for all $j=0,\ldots, i-1$;
     \item[b)]  $ g_{E^{(i)}}\leq g_E$ and   $\mathrm{rank}\, E^{(i)}\leq \mathrm{rank}\, E+ig_E$ for all $i\geq 1$ (the equalities hold for $i=1$);
     \item[c)] $g_{E^{(-i)}}\leq g_E$ and $\mathrm{rank}\, E^{(-i)}\geq \mathrm{rank}\, E-ig_E$ for all $i\geq 1$(the equalities hold for $i=1$);
      \item[d)] For each $g\geq g_E$, there exists a super-horizontal flag of holomorphic subbundles
     \begin{equation}\label{E_q}E_{-q}\subsetneq E_{-q+1}\subsetneq \ldots \subsetneq E_{-1}\subsetneq E_0=E,\end{equation}  such that $\mathrm{rank}\, E_{-i}=\mathrm{rank}\, E -ig$, where the integer $q\leq q_E$ is the quotient of the positive integer division of $\mathrm{rank}\, E$ by $g$: $\mathrm{rank}\, E=qg+r$.
   \end{enumerate}}
 \end{lem}
\begin{proof}
For example, since $E^{(-i-1)}=\ker \mathcal{A}_{E^{(-i)}}$, we have, for all $i\geq 0$,
\begin{equation*}\label{n1}
g_{E^{(-i)}}=\mathrm{rank}\,\mathrm{im}\mathcal{A}_{E^{(-i)}}=\mathrm{rank}\,\mathrm{coim}\mathcal{A}_{E^{(-i)}}= \mathrm{rank}E^{(-i)}- \mathrm{rank}E^{(-i-1)}.
\end{equation*}
On the other hand, since
   the image of $\mathcal{A}_{E^{(-i-1)}}$ is contained in $ E^{(-i)}\ominus E^{(-i-1)}$, for all $i\geq 0$, we also have
$g_{E^{(-i-1)}}\leq  \mathrm{rank}E^{(-i)}- \mathrm{rank}E^{(-i-1)}.$
  Hence, for all $i\geq  0$,
$g_{E^{(-i)}}\leq g_E.$

To construct a flag \eqref{E_q}, start by taking an arbitrary holomorphic subbundle $E_{-1}\subseteq E^{(-1)}$ with $\mathrm{rank}\, E_{-1}=\mathrm{rank}\, E -g\leq \mathrm{rank}\, E -g_E=\mathrm{rank}\,E^{(-1)}$. Clearly,
  \begin{equation}\label{ge}
  g_{E_{-1}}=\mathrm{rank}\,E_{-1}^{(1)}-\mathrm{rank}\,E_{-1}\leq \mathrm{rank}\, E -\mathrm{rank}\,E_{-1}= g.
  \end{equation}
  Hence $$\mathrm{rank}\, E_{-1}^{(-1)}= \mathrm{rank}\,E_{-1}-g_{E_{-1}}\geq \mathrm{rank}\,E-2g,$$
   and we see that there exists a holomorphic subbundle $E_{-2}$ of $E_{-1}^{(-1)}$ with $\mathrm{rank}\,E_{-2} =\mathrm{rank}\,E-2g$. Proceeding recursively we find after $q$ steps a super-horizontal flag of holomorphic subbundles \eqref{E_q}.
\end{proof}
The following construction is fundamental for our purposes:

 \begin{lem}\label{algc}
 \emph{Let  $T\subset E$ be two holomorphic subbundles of $\underline{\C}^n$.
 Fix a positive integer   $g$,  with $g\geq \max\{g_T,g_E\}$, and assume that, for some $i,j\geq 0$, we have $T^{(j)}\subset E^{(-i)}$. Given an integer $d$ with $\mathrm{rank}\,T^{(j)}< d<\mathrm{rank}\,E^{(-i)}$,  any holomorphic subbundle $F$ satisfying $T^{(j)}\subset F\subset E^{(-i)}$,
$\mathrm{rank}\,F=d$, and $g_F\leq g$, arises as follows:
 \begin{enumerate}
\item[a)]  Set
 $k_0  =\max\{k\,|\, d-kg> \mathrm{rank}\,T^{(j-k)}\}$ and  $r_0 =d-k_0g -\mathrm{rank}\,T^{(j-k_0)}$.
Choose $r_0$ linearly independent meromorphic sections $s_1,\ldots ,s_{r_0}$ of
 $E^{(-i-k_0)}$ so that the holomorphic vector bundle
 \begin{equation}\label{fk0} F_{-k_0}=T^{(j-k_0)}+\mathrm{Span}\{s_1,\ldots ,s_{r_0}\}\end{equation} has rank $d-k_0g$. Independently of the choice of these meromorphic sections, we have
 $g_{F_{-k_0}}\leq g$.
 \item[b)] Choose $r_1=d-(k_0-1)g-\mathrm{rank}F_{-k_0}^{(1)}$ meromorphic sections $s_{r_0+1},\ldots ,s_{r_0+r_1}$  of $E^{(-i-k_0+1)}$ so that
 the holomorphic vector subbundle
 \begin{equation*}\label{fk01}
 F_{-k_0+1}=F_{-k_0}^{(1)}+\mathrm{Span}\{s_{r_0+1},\ldots ,s_{r_0+r_1}\}
 \end{equation*} has rank $d-(k_0-1)g$. We have $g_{F_{-k_0+1}}\leq g$.
\item[c)] Repeat this procedure $k_0$ times to find a super-horizontal flag of holomorphic subbundles $F_{-k_0}\subsetneq \ldots\subsetneq  F_{-1}\subsetneq F_0=F$, with
\begin{equation}\label{k+l}
F_{-k_0+l}=F_{-k_0+l-1}^{(1)}+\mathrm{Span}\{s_{r_0+\ldots+r_{l-1}+1},\ldots ,s_{r_0+\ldots+r_{l-1}+r_l}\},\end{equation}
 $r_l=d-(k_0-l)g-\mathrm{rank}F_{-k_0+l-1}^{(1)}$ and $\mathrm{rank}\,F_{-k_0+l}=d-(k_0-l)g.$
  \end{enumerate}}
 \end{lem}
 \begin{proof}

Since $d< \mathrm{rank}\,E^{(-i)}$ and $g_{E^{(-i)}}\leq g_E\leq g$,  by Lemma \ref{E}
   inequalities  $$d-kg< \mathrm{rank}\,E^{(-i)}-kg_{E^{(-i)}}\leq \mathrm{rank}\,E^{(-i-k)}$$ hold  for each $k\geq 0$. Hence
we can always take $r_0\geq 0$  linearly independent meromorphic sections of $E^{(-i-k_0)}$ so that ${F_{-k_0}}$ defined by \eqref{fk0}   has rank $d-k_0g$.   We have to check now that
$g_{F_{-k_0}}\leq g$. By definition of $k_0$ we have
$d-(k_0+1)g\leq \mathrm{rank}\,T^{(j-k_0-1)}.$
Then,
 \begin{align*}\label{desg}g_{F_{-k_0}}&\leq  g_{T^{(j-k_0)}}+r_0 = g_{T^{(j-k_0)}} +d-k_0g-\mathrm{rank}T^{(j-k_0)} \\  &\leq g_{T^{(j-k_0)}}+g-(\mathrm{rank}T^{(j-k_0)}-\mathrm{rank}T^{(j-k_0-1)})= g_{T^{(j-k_0)}}+g-g_{T^{(j-k_0)}}=g.\end{align*}

Since $F_{-k_0}\subseteq \ker \mathcal{A}_{F_{-k_0+1}}$, then $g_{F_{-k_0+1}}\leq \mathrm{rank}\,{F_{-k_0+1}}-\mathrm{rank}\,{F_{-k_0}}=g$.
On the other hand, it is clear that $r_1\geq 0$. Hence the construction of item b) is possible and we can  proceed recursively until find a super-horizontal flag of holomorphic subbundles $F_{-k_0}\subsetneq \ldots\subsetneq  F_{-1}\subsetneq F_0=F$, with $F_{-k_0+l}$ given by \eqref{k+l}, where $F$ is certainly in the required conditions.

   Reciprocally, any $F$ as required  certainly arises in this way. In fact,  by Lemma \ref{E} there always exists  a super-horizontal flag of holomorphic subbundles $F_{-q}\subsetneq\ldots\subsetneq F_{-k_0}\subsetneq \ldots\subsetneq  F_{-1}\subsetneq F_0=F$, with  $k_0  =\max\{k\,|\, d-kg> \mathrm{rank}\,T^{(j-k)}\}$. We can choose such sequence so that $T^{(j-k_0)}\subsetneq F_{-k_0}$.

\end{proof}

Now we are in conditions to establish an algorithm to obtain all  $\mathrm{S}^1$-invariant extended solutions with values in a given $\Omega_\xi$.

\begin{thm}
\emph{Fix $\xi\in \mathfrak{I}(\mathrm{U}(n))$ and consider the corresponding flag \eqref{flag1}. Set $d_i=\dim A^\xi_i$ and $h_i=d_{i+1}-d_{i}$. Any super-horizontal flag of holomorphic vector subbundles
\begin{equation}\label{flag2}
\{0\}=A_{-r-1}\subsetneq A_{-r} \subseteq \ldots \subseteq A_{k-1} \subsetneq A_k=\underline{\C}^n
 \end{equation}
 with $\mathrm{rank}\, A_i=d_i$  arises as follows:
 \begin{enumerate}
 \item[a)] Set $l=\min\{h_i\,|\, i=-r-1,\ldots,k-1\}$ and $m=\max\{i\,|\, l=h_i\}$. Apply  Lemma \ref{algc} (with $T=\{0\}$, $E= \underline{\C}^n$, $d= d_m$ and $g=l$) to find $A_m$.
\item[b)] Set
$l_1=\min\{h_i\,|\,  -r-1 \leq i <m \}$, $m_1=\max\{i\,|\, l_1=h_i, -r-1 \leq i <m\}$. Apply  Lemma \ref{algc} (with $T=\{0\}$, $E= A_m$, $d= d_{m_1}$ and $g=l_1$) to find $A_{m_1}\subseteq {A}_m^{(m_1-m)}.$
 \item[c)] Set $l_{\hat{1}}=\min\{h_i\,|\,  m<i\leq k-1 \}$ and $m_{\hat{1}}=\max\{i\,|\, l_{\hat{1}}=h_i, m<i\leq k-1\}$, and apply  Lemma \ref{algc}
(with  $T= A_m$, $E= \underline{\C}^n$, $d= d_{m_{\hat 1}}$ and $g=l_{\hat {1}}$)
 to find $A_{m_{\hat{1}}}\supseteq A_m^{(m_{\hat{1}}-m)}.$
 \item[d)] Proceed recursively until  obtain  a flag of the form \eqref{flag2}.
 \end{enumerate}}
\end{thm}

\begin{rem}
  In \cite{burstall_guest_1997}, the authors introduce a method to obtain super-horizontal flags of holomorphic subspaces associated to a given element $\xi\in\mathfrak{I}'(G)$. However, their method involves integration of meromorphic functions.
\end{rem}

Finally, take a super-horizontal flag of holomorphic vector subbundles \eqref{flag2} and the corresponding $\mathrm{S}^1$-invariant extended solution $W_A$.
Take a meromorphic frame $s_1,\ldots,s_{d_{k-1}}$ of $A_{k-1}$ such that, for each $i\in \{-r,\ldots,k-1\}$, $s_1,\ldots,s_{d_i}$ is a meromorphic frame of $A_i$ and $s_1,\ldots,s_{d_i},s_{d_i+1},\ldots,s_{d_i+g_i}$ is a meromorphic frame of  $A_i^{(1)}$. The extended solution
$W$, with values in $U_\xi(\mathrm{U}(n))$   and $u_\xi(W)=W_A$, have Frenet frames of the form
\begin{align}\label{frenetframe}
\nonumber X&=  \mathrm{Span}\{s_1\lambda^{-r}+ w_1\lambda^{-r+1},\ldots, s_{d_{-r}}\lambda^{-r}+ w_{d_{-r}}\lambda^{-r+1}\}\\&+\sum_{i=-r}^{k-2}\mathrm{Span}\{s_{d_i+g_i+1}\lambda^{i+1}+w_{d_i+g_i+1}\lambda^{i+2},\ldots,s_{d_{i+1}}\lambda^{i+1}+w_{d_{i+1}}\lambda^{i+2} \};
\end{align}
where, for each $j\in\{1,\ldots,d_{k-1}\}$, $w_j$ is a meromorphic section of
 $M\times H_+^n/\lambda^{r+k}H_+^n$. However, in the general case, these meromorphic sections $w_j$ can not be chosen arbitrarily.  For example, if $s_1$ is a constant section, $w'_1\lambda^{-r+2}$ becomes a section of $W$. So we have to impose that $p_{0}(w'_1)$, whith $p_0$ the projection defined by \eqref{pi}, has no orthogonal component onto $A^\perp_{-r+2}$.  In sections \ref{3s} and \ref{45} we shall discuss in  detail some examples.


\section{Extended solutions in $\Omega \mathrm{SU}(n)$}

\subsection{Grassmannian model for $\Omega \mathrm{SU}(n)$}

Consider the exterior product $\wedge$ of vectors in $\C^n$
and extend it to  $H^n$ as follows: if $f,g\in H^n$, then $(f\wedge g)(\lambda)=f(\lambda)\wedge g(\lambda)$.
The loop group $\Omega \mathrm{U}(n)$ acts on $\wedge^nH^n$ in the natural way:
$$\gamma(f_1\wedge\ldots\wedge f_n):=\gamma f_1\wedge\ldots\wedge \gamma f_n=\det (\gamma)(f_1\wedge\ldots\wedge f_n).$$
The Grassmannian model of $\Omega \mathrm{SU}(n)$ is given by:
\begin{prop}
\emph{A subspace $W\in \mathrm{Gr}^n$ corresponds to
a loop in $\mathrm{SU}(n)$ if, and only if, it belongs to \bdm
\mathrm{Gr}(\mathrm{SU}(n))=\{W\in \mathrm{Gr}^n\,|\,
\,\wedge^n W=\wedge^n H_+^n\}.\edm}
\end{prop}
\begin{proof}
If $\gamma\in \Omega \mathrm{SU}(n)$, then it is clear that $\wedge^n W=\wedge^n H_+^n$, since $\Omega \mathrm{SU}(n)$ acts trivially on the $\wedge^nH^n$.  Conversely, suppose that $\wedge^n W=\wedge^n H_+^n$. For each $\lambda\in \mathrm{S}^1$, consider the isomorphism given by
 the evaluation map
at $\lambda$, ${\rm{ev}}_\lambda:W\ominus \lambda W \to \mathbb{C}^n$.
Set $\gamma(\lambda)={\rm{ev}}_\lambda\circ {\rm{ev}}^{-1}_1$, which is a loop of $\Omega \mathrm{U}(n)$ and,  by Remark \ref{rem}, verifies  $W=\gamma  H_+^n$.
By hypothesis
$\wedge^n (W\ominus \lambda W)\subset \wedge^n H_+^n.$
Hence, ${\rm{ev}}_1^{-1}(e_1)\wedge \ldots \wedge {\rm{ev}}_1^{-1}(e_n)\in\wedge^n H_+^n$. Since
\begin{align*}
\det(\gamma)(e_1\wedge\ldots\wedge e_n)&= \gamma(e_1)\wedge\gamma(e_2)\wedge\ldots\wedge \gamma(e_n)={\rm{ev}}_\lambda\circ {\rm{ev}}_1^{-1}(e_1)\wedge \ldots \wedge {\rm{ev}}_\lambda\circ {\rm{ev}}_1^{-1}(e_n)\\&={\rm{ev}}_\lambda( {\rm{ev}}_1^{-1}(e_1)\wedge \ldots \wedge  {\rm{ev}}_1^{-1}(e_n) ) ,
\end{align*}
it follows that $\det(\gamma)$ is in $H_+^1$.

Now, since $\wedge^n\gamma  H_+^n=\wedge^n H_+^n$, we also have  $\wedge^n H_+^n=\wedge^n \gamma^{-1} H_+^n$. Hence, by the same argument as above,  $\det(\gamma)^{-1}$ is in  $H_+^1$.
On the other hand, the fact that $\gamma$ takes values in $\mathrm{U}(n)$ implies that $\det(\gamma)^{-1}=\overline{\det(\gamma)}$, which means that $\det(\gamma)^{-1}$ is also in  $H_-^1$.
This is possible if and only if $\det(\gamma)$ is constant. Since $\gamma(1)=e$, we must have $\det(\gamma)=1$.
\end{proof}

\begin{prop}
  If $\xi\in \mathfrak{I}(\mathrm{SU}(n))\subset \mathfrak{I}(\mathrm{U}(n))$, then $U_\xi(\mathrm{SU}(n))=U_\xi(\mathrm{U}(n))$.
\end{prop}
\begin{proof}
  Let $\gamma\in U_\xi(\mathrm{U}(n))$. Then $\gamma H^n_+=\Psi \gamma_\xi H^n_+$ for some $\Psi\in \Lambda^+_{\mathrm{alg}}\mathrm{U}(n)$. Hence
  \begin{equation}\label{wedge}
  \wedge^n\gamma H^n_+ = \wedge^n\Psi\gamma_\xi H^n_+= \wedge ^n\Psi W_\xi=\det (\Psi)\wedge^n W_\xi=\det (\Psi)\wedge^n H^n_+.
  \end{equation}
 Since $\Psi\in \Lambda^+_{\mathrm{alg}}\mathrm{U}(n)$, $\det (\Psi)$ is polynomial in $\lambda$, hence  $\wedge^n\gamma H^n_+\subseteq \wedge^n H^n_+.$ Conversely, from \eqref{wedge} we see that
  \begin{equation*}\label{wedge1}
  \wedge^n H^n_+=\det (\Psi^{-1})\wedge^n\gamma H^n_+;
  \end{equation*}
 and since  $\det (\Psi^{-1})$ is also  polynomial in $\lambda$,  we conclude that $\wedge^n H^n_+ \subseteq  \wedge^n\gamma H^n_+.$
\end{proof}

In particular, if $\xi\in \mathfrak{I}(\mathrm{SU}(n))$, all the extended solutions $W:M\setminus D\to U_\xi(\mathrm{SU}(n))$ arise from a Frenet frame of the form \eqref{frenetframe} without  any further restriction on the choice of the meromorphic data.

\subsection{Canonical elements of $\mathrm{SU}(n)$}

Let $E_i$ be the $(n\times n)$-matrix whose  $(i,i)$ entry is $\sqrt{-1}$ and  whose other entries are all $0$.
The algebra of diagonal matrices
$$\mathfrak{t}^\C=\big\{\sum a_iE_i:\, a_i\in \C, \, \sum a_i=0\big\}$$
is a
Cartan subalgebra  of $\mathfrak{su}(n)^\C=\mathfrak{sl}(n,\C)$.
Let $L_i$ in the dual of $\mathfrak{t}^\C$ be defined by $L_i(\sum a_jE_j)=\sqrt{-1}a_i$.
The corresponding set of roots $\Delta\in \sqrt{-1} \mathfrak{t}^*$ is given by $\Delta=\{L_i-L_j:\, i,j=1,\ldots,n\}$ and $\Delta^+=\{L_i-L_j:\,i<j\}$  is a set of positive roots. The positive simple roots are then the roots of the form  $\alpha_i=L_i-L_{i+1}$, with $i=1,\ldots,n-1$, and the dual basis of $\mathfrak{t}$ is formed by the matrices
\begin{align*}
H_i=\frac{n-i}{n}E_1+\ldots+ \frac{n-i}{n}E_i- \frac{i}{n}E_{i+1}-\ldots-\frac{i}{n}E_{n}.
\end{align*}

The Lie group $\mathrm{SU}(n)$ is precisely the simply connected Lie group with Lie algebra $\mathfrak{su}(n)$ and its center is $\mathbb{Z}_n$. The integer lattice $\mathfrak{I}(\mathrm{SU}(n)/\mathbb{Z}_n)$ is simply $\{\sum n_iH_i:\, n_i\in \mathbb{Z}\}$ and its $I$-canonical elements with respect to $\Delta^+$ are the sums $\sum_{i\in I}H_i$, with $I\subseteq \{1,\ldots,n-1\}$. The $I$-canonical elements of $\mathfrak{\mathrm{SU}}(n)$ are not so easy to identify. We need to find the integral combinations of the elements $H_i$ which are in $\mathfrak{I}'(\mathrm{G})\cap \mathfrak{C}_{I}$ (that is, elements which are simultaneously integral combinations of the elements $H_i$ and of the elements $E_i$) and are maximal with respect to $\preceq$.    For example,
when $n$ is odd, it is easy to check that
$\xi=H_1+H_2+\ldots + H_{n-1}$ is the unique $[n-1]$-canonical element of $\mathrm{SU}(n)$ with respect to $\Delta^+$, where $[p]=\{1,\ldots,p\}$. But
when $n>2$ is even there always exist more than one $[n-1]$-canonical element. The following lemma is useful in order to describe all canonical elements of $\mathrm{SU}(n)$:

\begin{lem}\label{chis} \emph{The integer lattice $\mathfrak{I}(\mathrm{SU}(n))$ is invariant with respect to the linear isomorphism $\chi_1:\mathfrak{t}^\C\to \mathfrak{t}^\C$ defined by $\chi_1(H_i)=H_{n-i}$, with $i\in[n-1]$.
When $n=2m+1$ is odd, $\mathfrak{I}(\mathrm{SU}(n))$ is also invariant with respect to the linear isomorphism $\chi_2:\mathfrak{t}^\C\to \mathfrak{t}^\C$ defined by $\chi_2(H_i)=H_{2i}$ and $\chi_2(H_{n-i})=H_{n-2i}$ if  $i\in\{1,\ldots,m\}$.}
\end{lem}
\begin{proof}
  As we have observed before, an element of $\mathfrak{t}$ is in $\mathfrak{I}(\mathrm{SU}(n))$  if and only if its coefficients in $E_i$ are integers. Hence, taking account that $H_i=E_1+\ldots+E_i-\frac{i}{n}(E_1+\ldots+E_n)$,  an integer linear combination $\sum_{i=1}^{n-1}n_iH_i$ is in $\mathfrak{I}(\mathrm{SU}(n))$ if and only if $\sum_{i=1}^{n-1}\frac{in_i}{n}$ is an integer number, and this happens if and only if $\sum_{i=1}^{n-1}\frac{(n-i)n_i}{n}$ is integer. Hence  $\mathfrak{I}(\mathrm{SU}(n))$ is invariant with respect to $\chi_1$.

Assume now that $n=2m+1$. In this case,
\begin{equation}\label{nis}
\sum_{i=1}^{n-1}\frac{in_i}{n}=\sum_{i=1}^{m}\frac{in_i}{n}+ \sum_{i=1}^{m}\frac{(n-i)n_{n-i}}{n}=\sum_{i=1}^{m}\frac{in_i}{n}- \sum_{i=1}^{m}\frac{in_{n-i}}{n}+\sum_{i=1}^{m}n_{n-i}.
\end{equation}
On the other hand, if we set
\begin{equation*}
\sum_{i=1}^{n-1}n'_iH_i=\chi_2\big(\sum_{i=1}^{n-1}n_iH_i\big)=\sum_{i=1}^mn_iH_{2i}+\sum_{i=1}^mn_{n-i}H_{n-2i},
\end{equation*}
we get
\begin{equation}\label{nis1}
\sum_{i=1}^{n-1}\frac{in'_i}{n}=\sum_{i=1}^m\frac{2i n_i}{n}+\sum_{i=1}^m\frac{(n-2i)n_{n-i}}{n}=2\sum_{i=1}^m\frac{i n_i}{n}-2\sum_{i=1}^m\frac{in_{n-i}}{n}+\sum_{i=1}^{m}n_{n-i}.
\end{equation}
Comparing \eqref{nis} with \eqref{nis1} we conclude that $\sum_{i=1}^{n-1}\frac{in'_i}{n}\in \mathbb{Z}$ if $\sum_{i=1}^{n-1}\frac{in_i}{n}\in \mathbb{Z}$, that is  $\mathfrak{I}(\mathrm{SU}(n))$ is invariant with respect to $\chi_2$.
\end{proof}

For each $i\in \{1,\ldots,n-1\}$, let  $m_i$ be the least positive integer which makes $m_iH_i$  and integral combination of the elements $E_i$. Since $H_i=E_1+\ldots+E_i-\frac{i}{n}(E_1+\ldots+E_n)$, $m_i$  is precisely the denominator of the irreducible fraction equivalent to $\frac{i}{n}$. The canonical elements should then be sought among the elements of the finite set formed by the integral combinations $\sum_{i=1}^{n-1}n_iH_i$, with $n_i\in\{0,\ldots,m_i\}$, which are simultaneously integral combinations of the elements $E_i$. For  general $n$ and $I\subseteq\{1,\ldots,n-1\}$ it is too hard to list all the $I$-canonical elements.
Next we will describe in detail the situation for the lower dimensional cases. We shall denote  by $\pi_i$ the orthogonal projection of $\C^n$  onto the one-dimensional vector subspace of $\C^n$ generated by the vector $e_i$.
%
%
%
%
%
%
%
%
%
%
%
%
%
%
%
%
%

\subsection{The case $\mathrm{SU}(2)$}

In this case there is a unique simple root $\alpha_1$ with dual $H_1=\frac12(E_1-E_2)$, which does not belong to the integer lattice $\mathfrak{I}(\mathrm{SU}(2))$. Consequently    $\xi=2H_1$ is the unique non-trivial canonical element -- the corresponding homomorphism is $\gamma_\xi(\lambda)=\lambda^{-1}\pi_1+\lambda \pi_1^\perp$.    If $W:M\setminus D\to U_\xi(\mathrm{SU}(n))$ is a complex extended solution, then the corresponding $\mathrm{S}^1$-invariant solution is given by $u_\xi(W)=\lambda^{-1} A+A+\lambda H_+^2$, where $A$ is a holomorphic subbundle of $\underline{\C}^2$. It follows from the super-horizontality property that $A$ is a constant bundle. Hence,
we have
$W=L+A+\lambda H_+^2$, where $L$ is a holomorphic line bundle of $A\lambda^{-1}+A^\perp$, with $p_{-1}(L)\neq 0$ off a discrete set of points, where $p_{-1}$ is defined as in \eqref{pi}. That is, any harmonic map of finite uniton number from $M$ into $\mathrm{SU}(2)\simeq \mathrm{S}^3$ arises from a constant direction  $A$ and a holomorphic line bundle of  $\lambda^{-1} A+A^\perp\simeq \underline{\C}^2$. This agrees with the well known result by Calabi \cite{calabi_1967} that asserts that any locally minimal immersion of a surface in an odd dimensional sphere $S^{2m-1}$ is contained in a hyperplane of $\mathbb{R}^{2m}$. In particular, no harmonic map of finite uniton number from $M$ into $\mathrm{SU}(2)\simeq \mathrm{S}^3$ is full. This means that any such harmonic map takes values in a unit two-dimensional sphere $\mathrm{S}^2\simeq \C \mathrm{P}^1$, that is, it corresponds to an holomorphic line bundle of $\underline{\C}^2$.

%
%
%
%
%
%
%
%
%
%

\subsection{The case $\mathrm{SU}(3)$}\label{3s}

We have two simple roots, $\alpha_1$ and $\alpha_2$, and  three non-trivial canonical elements:
\begin{align*}
 \xi_1=H_1+H_2=E_1-E_3;\,\,\,\,\,
\xi_2=3H_1=2E_1-E_2-E_3;\,\,\,\,\, \xi_3=3H_2=E_1+E_2-2E_3.
\end{align*}
 The corresponding homomorphisms are given by
\begin{align*}
 \gamma_{\xi_1}(\lambda)=\lambda^{-1}\pi_3+\pi_2+\lambda \pi_1;\,\,\,\gamma_{\xi_2}(\lambda)=\lambda^{-1}\left(\pi_2+\pi_3\right)+\lambda^2\pi_1;\,\,\,\gamma_{\xi_3}(\lambda)=\lambda^{-2}\pi_3+\lambda(\pi_2+\pi_1).
\end{align*}

If $W_{\xi_1}:M\setminus D\to U_{\xi_1}(\mathrm{SU}(3))$ is a complex extended solution, then the corresponding $\mathrm{S}^1$-invariant solution is given by
$u_{\xi_1}(W_{\xi_1})=\lambda^{-1} B_3+\left(B_2\oplus B_3\right)+\lambda H^3_+,$
where $B_3$ is a holomorphic line subbundle of the holomorphic vector bundle $B_2\oplus B_3$ of rank $2$.
 In order to construct all such extended solutions, and taking account the results of section \ref{construction}, we start with a meromorphic section $s_3$ of $\underline{\C}^3$ and set $B_3=\mathrm{Span}\{s_3\}$.
If $B_3$ is not constant, we define $B_{2}=B_3^{(1)}\ominus B_3$, take an arbitrary holomorphic section $w_3$ of $\underline{\C}^3$ and set $X_{\xi_1}=\mathrm{Span}\{\lambda^{-1}s_3+w_3\}$. If $B_3$ is constant, we take an arbitrary meromorphic section $s_2$. By adding a constant if necessary, $s_2$ and $s_3$ are linearly independent and we set $B_2=\mathrm{Span}\{s_2,s_3\}\ominus B_3$. Take an arbitrary holomorphic section $w_3$ of $\underline{\C}^3$ and set $X_{\xi_1}=\mathrm{Span}\{\lambda^{-1}s_3+w_3,s_2\}$. In both cases, $X_{\xi_1}$ is a Frenet frame for an extended solution
$W_{\xi_1}:M\setminus D\to U_{\xi_1}(\mathrm{SU}(3))$.

If $W_{\xi_2}:M\setminus D\to U_{\xi_2}(\mathrm{SU}(3))$ is a complex extended solution, then the corresponding $\mathrm{S}^1$-invariant solution is given by
$$u_{\xi_2}(W_{\xi_2})=\lambda^{-1}(B_2\oplus B_3)+(B_2\oplus B_3)+\lambda(B_2\oplus B_3)+\lambda^2 H_+^3.$$
By the super-horizontality property, $B_2\oplus B_3$ is constant, and consequently $B_1$, the orthogonal complement of $A_2\oplus A_3$, is also constant. In order to construct all such extended solutions, fix a two-dimensional subspace with  basis elements $s_2$ and $s_3$. Take arbitrary meromorphic sections $w_2$ and $w_3$ of $\underline{\C}^3+\lambda \underline{\C}^3$ with $\pi_{B_1}(p_0(w_2))$ and $\pi_{B_1}(p_0(w_3))$ constants, where $p_0$ is the projection defined by \eqref{pi}. Then $X_{\xi_2}=\mathrm{Span}\{s_2\lambda^{-1}+w_2, s_3\lambda^{-1}+w_3  \}$  is a Frenet frame for an extended solution
$W_{\xi_2}:M\setminus D\to U_{\xi_2}(\mathrm{SU}(3))$, and all such extended solutions arise in this way.
We observe that, taking account Lemma 3.17 and Proposition 3.18 of \cite{svensson_wood_2010}, the corresponding harmonic map has \emph{uniton number} one, in the sense that it admits an extended solution with values in $\Omega\mathrm{U}(3)$ of the form $\pi_V+\lambda \pi_V^\perp$, with $V$ a holomorphic subbundle of $\underline{\C}^3$.
The case $W_{\xi_3}$ is similar.

\subsection{The cases $\mathrm{SU}(4)$ and $\mathrm{SU}(5)$}\label{45}

 Table \ref{cococo} shows all the non-trivial canonical elements of $\mathrm{SU}(4)$ and $\mathrm{SU}(5)$ up to the symmetries $\chi_1,\chi_2$ of Lemma \ref{chis}.

\begin{table}[!htb]
\begin{tabular}{c|c|c|c|c} $\SU(n)$ & $|I|=n-1$ & $|I|=n-2$ & $|I|=n-3$ & $|I|=n-4$\\ \hline $n=4$ &$H_1+2H_2+H_3$ & $2H_1+H_2$ & $4H_1$ & \\ & $3H_1+H_2+H_3$  & $H_1+H_3$ & $2H_2$ & \\ \hline $n=5$ & $H_1+H_2+H_3+H_4$ &$H_1+H_2+4H_3$ & $H_1+2H_2$ & $5H_1$\\ & & $H_1+3H_2+H_3$ & $3H_1+H_2$ & \\ & &  $2H_1+H_2+2H_3$  & $H_1+H_4$ &\\ & & $3H_1+2H_2+H_3$ & \\ & & $5H_1+H_2+H_3$   & \end{tabular}

\vspace{.10in}
\caption{Canonical elements for $\mathrm{SU}(4)$ and $\mathrm{SU}(5)$.}\label{cococo}
\end{table}

We describe  how to construct, for
$\xi_1=H_1+2H_2+H_3=2E_1+E_2-E_3-2E_4,$
all  the extended solutions $W_{\xi_1}:M\setminus D\to U_{\xi_1}(\mathrm{SU}(4))$.
We have $\gamma_{\xi_1}=\lambda^{-2}\pi_4+\lambda^{-1}\pi_3+\lambda\pi_2+\lambda^2\pi_1$, and, consequently,
$$u_{\xi_1}(W_{\xi_1})=\lambda^{-2} B^1_4+\lambda^{-1}\left(B^1_3\oplus B^1_4\right)+\left(B^1_3\oplus B^1_4\right)+\lambda\left(B^1_2\oplus B^1_3\oplus B^1_4\right)+ \lambda^2H_+^4, $$
where each vector subbundle $B^1_i$ has rank one.
The harmonic map associated to this $\mathrm{S}^1$-invariant extended solution is given by
$$\varphi_1=\pi_{B^1_1\oplus B^1_4}-\pi_{B^1_2\oplus B^1_3}.$$
By super-horizontality, $B^1_3\oplus B^1_4$ is a constant bundle. So, in order to construct all such extended solutions, we start by fixing a two-dimensional  vector subspace $V$ of $\C^4$ generated by constant vectors $u$ and $v$. Next take  meromorphic sections:
 \begin{enumerate}
  \item $s_1$ of $\underline{V}$,  and set $B^1_4=\mathrm{Span}\{s_1\}$ and $B^1_3=\underline{V}\ominus B^1_4$;
  \item  $s_3$ of $\underline{V}^\perp$, and set $B^1_2=\mathrm{Span}\{s_3\}$ and $B^1_1=\underline{V}^\perp\ominus B^1_2$;
  \item $w_1$ of $M\times  H^4_+/\lambda^3 H_+^4$.
 \end{enumerate} If $B^1_4$ is not constant, we can write
 $s_1''=g_1s_1+g_2s'_1$ for some meromorphic functions $g_1$ and $g_2$ on $M$, with $g'_1g_2-g'_2g_1\neq 0$. In this case,
$X=\mathrm{Span}\{\lambda^{-2}s_1+\lambda^{-1}w_1, \lambda s_3\}$ is a Frenet frame for an extended solution with values in $U_{\xi_1}(\mathrm{SU}(4))$ if and only if
$$\pi_{B^1_1}\circ p_0(w_1''-g_1w_1-g_2w_1')=0.$$

 For
$\xi_2=3H_1+H_2+H_3=3E_1-E_3-2E_4,$
we have $\gamma_{\xi_2}=\lambda^{-2}\pi_4+\lambda^{-1}\pi_3+\pi_2+\lambda^3\pi_1$, and, consequently,
\begin{align*}
u_{\xi_2}(W_{\xi_2})=\lambda^{-2} B^2_4&+\lambda^{-1}\left(B^2_3\oplus B^2_4\right)\\&+\left(B^2_2\oplus B^2_3\oplus  B^2_4\right)+\lambda\left(B^2_2\oplus B^2_3\oplus B^2_4\right)+ \lambda^2\left(B^2_2\oplus B^2_3\oplus B^2_4\right)+\lambda^3 H_+^4. \end{align*}
The harmonic map associated to this $\mathrm{S}^1$-invariant extended solution is given by
$$\varphi_2=\pi_{B^2_2\oplus B^2_4}-\pi_{B^2_1\oplus B^2_3}.$$
 Although $\varphi_1$ and $\varphi_2$ are both of the form $\pi_E-\pi_E^\perp$, with $E$ a rank two vector subbundle of $\underline{\C}^4$, these vector bundles exhibit distinct geometrical behaviours. For example, whereas $E=B^2_2\oplus B^2_4$ has always constant $(2)$-osculating bundle,  both $E=B^1_1\oplus B^1_4$ and $E^\perp=B^1_2\oplus B^1_3$ can have
 non-constant  $(2)$-osculating bundle.

%
%

\subsection{Symmetric canonical elements of $\mathrm{SU}(n)$}

All the compact inner symmetric spaces of $\mathrm{SU}(n)$ are complex grassmannians $\mathrm{Gr}(k,n)$.
  The embedding $\iota$ of $\mathrm{Gr}(k,n)$ as a connected component of $\sqrt{e}$ is given by $\iota (V) =\pi_V - \pi_V^\perp$, where $V\in \mathrm{Gr}(k,n)$ is $k$-dimensional subspace of $\C^n$.  There exists no non-trivial symmetric canonical element  for $\mathrm{SU}(2)$.  Table \ref{xixixi} presents all  non-trivial symmetric canonical elements for $\mathrm{SU}(n)$, with $n=3,4,5$, up to the symmetries $\chi_1,\chi_2$.
   As before, for each $i\in \{1,\ldots,n-1\}$, let $m_i$ be the least positive integer which makes $m_iH_i$  and integral combination of the elements $E_i$.  The symmetric canonical elements should then be sought among the elements of the finite set formed by the integral combinations $\sum_{i=1}^{n-1}n_iH_i$, with $n_i\in\{0,\ldots,2m_i-1\}$, which are simultaneously integral combinations of the elements $E_i$.

   We use the usual hermitian inner product on $\C^n$ to identify $\mathrm{Gr}(k,n)$ with $\mathrm{Gr}(n-k,n)$. It is easy to check that, for $\xi\in \mathfrak{I}(\mathrm{SU}(n))$, $N_{\xi}=N_{\chi_1(\xi)}$.
However, in general, the symmetric space $N_{\chi_2(\xi)}$ does not coincide with $N_\xi$.  For example, in $\mathrm{SU}(5)$ the two following situations can occur: for $\xi=5H_1$, we have $\chi_2(\xi)=5H_2$, $N_\xi= \mathrm{Gr}(1,5)$ and $N_{\chi_2(\xi)}= \mathrm{Gr}(2,5)$; on the other hand, for $\eta=3H_1+H_2+5H_3$, we have $N_{\eta}=N_{\chi_2(\eta)}=\mathrm{Gr}(2,5)$.

\begin{table}[!htb]
\begin{tabular}{ c| c| c| c| c }
 $ \mathrm{Gr}(k,n)$ & $|I|=n-1$ & $|I|=n-2$ & $|I|=n-3$ & $|I|=n-4$\\ \hline
$k=1,n=3$ & $H_1+H_2$ & $3H_1$ & &\\ & $4H_1+H_2$ & & & \\ \hline
$k=2,n=4$  & $3H_1+H_2+H_3$ & $2H_1+H_2$ &  & \\ &  &$ H_1+H_3$ &  & \\

 \hline $k=1, n=5$ & $4H_1+2H_2+H_3+H_4$&$ H_1+H_2+4H_3$  & $H_1+2H_2$ & $5H_1$\\ & &&  $H_1+7H_2$ & \\ && & $3H_1+H_2$&\\ & && $H_1+6H_4$    \\

 \hline $k=2, n=5$ & $H_1+H_2+H_3+H_4$ & $H_1+H_2+9H_3$  & $4H_1+3H_2$ & \\ & $2H_1+3H_2+H_3+H_4$ &$H_1+3H_2+H_3$ & $8H_1+H_2$& \\ &$H_1+H_2+H_3+6H_4$ & $H_1+8H_2+H_3$ & $ H_1+H_4$ &\\ & & $2H_1+H_2+2H_3$& \\ & & $3H_1+H_2+5H_3$  & \\ & & $3H_1+2H_2+H_3$   & \\ & &  $5H_1+H_2+H_3$ &
 \end{tabular}
\vspace{.10in}
\caption{Symmetric canonical elements for $\mathrm{SU}(n)$, with $n\leq 5$.} \label{xixixi}
\end{table}

We describe  how to construct, for $n=4$, $k=2$ and
$\xi_1=2H_1+H_2=2E_1-E_3-E_4,$
all  the extended solutions $W_{\xi_1}:M\setminus D\to U^{\mathcal{I}}_{\xi_1}(\mathrm{SU}(4))$.
We have $\gamma_{\xi_1}=\lambda^{-1}(\pi_3+\pi_4)+\pi_2+\lambda^2\pi_1$ and, consequently,
$$u_{\xi_1}(W_{\xi_1})=\lambda^{-2} (B^1_4\oplus B^1_3)+(B^1_4\oplus B^1_3\oplus B^1_2)+\lambda (B^1_4\oplus B^1_3\oplus B^1_2)+ \lambda^2H_+^4, $$
where each vector subbundle $B^1_i$ has rank one.
The harmonic map associated to this $\mathrm{S}^1$-invariant extended solution is given by
$\varphi_1=\pi_{B^1_1\oplus B^1_2}-\pi_{B^1_3\oplus B^1_4}.$
So, take meromorphic sections $s_1,s_2,w_1,w_2$ of $\underline{\C}^4$ and set $B_3\oplus B_4=\mathrm{Span}\{s_1,s_2\}$. Assuming that this vector bundle is not constant, $X_{\xi_1}=\mathrm{Span}\{\lambda^{-1}s_1+\lambda w_1,\lambda^{-1}s_2+\lambda w_2\}$ will be a Frenet frame for an extended solution with values in $U^{\mathcal{I}}_{\xi_1}(\mathrm{SU}(4))$. Moreover, all such extended solutions, with $B_3\oplus B_4$ not constant, arise in this way.

Consider also the case $n=4$, $k=2$ and  $\xi_2=H_1+H_3=E_1-E_4$. We have $\gamma_{\xi_2}=\lambda^{-1}\pi_4+(\pi_3\oplus\pi_2)+\lambda\pi_1$ and, consequently,
$u_{\xi_2}(W_{\xi_2})=\lambda^{-1} B^2_4+B^2_4\oplus B^2_3\oplus B^2_2+\lambda H_+^4. $
The harmonic map associated to this $\mathrm{S}^1$-invariant extended solution is given by
$\varphi_2=\pi_{B^2_2\oplus B^2_3}-\pi_{B^2_1\oplus B^2_4}.$ Observe that, in this case, we only have $S^1$-invariant extended solutions since $U^{\mathcal{I}}_{\xi_2}(\mathrm{SU}(4))=\Omega_{\xi_2}$.

\providecommand{\bysame}{\leavevmode\hbox to3em{\hrulefill}\thinspace}
\providecommand{\MR}{\relax\ifhmode\unskip\space\fi MR }
\providecommand{\MRhref}[2]{%
  \href{http://www.ams.org/mathscinet-getitem?mr=#1}{#2}
}
\providecommand{\href}[2]{#2}

\end{document}